\documentclass[12pt,reqno]{amsart}
\usepackage{enumerate, latexsym, amsmath, amsfonts, amssymb, amsthm, color}
\def\pmod #1{\ ({\rm{mod}}\ #1)}

\def\bg{\bigg}
\def\({\bg(}
\def\){\bg)}

\def\gs{\geqslant}

\def\bi{\binom}

\def\eq{\equiv}

\def\Ack{\medskip\noindent {\bf Acknowledgment}}
\theoremstyle{plain}
\newtheorem{theorem}{Theorem}

\newtheorem{lemma}{Lemma}

\newtheorem{conjecture}{Conjecture}
\theoremstyle{definition}

\theoremstyle{remark}

 \vspace{4mm}

\begin{document}
\title
[{On two congruence conjectures of Z.-W. Sun}]
{On two congruence conjectures of Z.-W. Sun involving Franel numbers}

\author
[Guo-Shuai Mao and Yan Liu] {Guo-Shuai Mao and Yan Liu}

\address {(Guo-Shuai Mao) Department of Mathematics, Nanjing
University of Information Science and Technology, Nanjing 210044,  People's Republic of China\\
{\tt maogsmath@163.com  } }

\address{(Yan Liu) Department of Mathematics, Nanjing
University of Information Science and Technology, Nanjing 210044,  People's Republic of China\\
{\tt 1325507759@qq.com  } }

\keywords{Congruences; franel numbers; $p$-adic Gamma function; Gamma function.
\newline \indent {\it Mathematics Subject Classification}. Primary 11A07; Secondary 05A10; 05A19; 33B15.
\newline \indent The first author is the corresponding author. This research was supported by the Natural Science Foundation of China (grant 12001288).}

 \begin{abstract} In this paper, we mainly prove the following conjectures of Z.-W. Sun \cite{S13}:
 Let $p>2$ be a prime. If $p=x^2+3y^2$ with $x,y\in\mathbb{Z}$ and $x\equiv1\pmod 3$, then
 $$
 x\equiv\frac14\sum_{k=0}^{p-1}(3k+4)\frac{f_k}{2^k}\equiv\frac12\sum_{k=0}^{p-1}(3k+2)\frac{f_k}{(-4)^k}\pmod{p^2},
 $$
 and if $p\equiv1\pmod3$, then
 $$
 \sum_{k=0}^{p-1}\frac{f_k}{2^k}\equiv\sum_{k=0}^{p-1}\frac{f_k}{(-4)^k}\pmod{p^3},
 $$
 where $f_n=\sum_{k=0}^n\binom{n}k^3$ stands for the $n$th Franel number.
\end{abstract}

\maketitle

\section{Introduction}
\setcounter{lemma}{0}
\setcounter{theorem}{0}
\setcounter{corollary}{0}
\setcounter{remark}{0}
\setcounter{equation}{0}
\setcounter{conjecture}{0}
In 1894, Franel \cite{F} found that the numbers
$$
f_n=\sum_{k=0}^n\binom{n}k^3\ \ (n=0,1,2,\ldots)
$$
satisfy the recurrence relation (cf. \cite[A000172]{SL}):
$$
(n+1)^2f_{n+1}=(7n^2+7n+2)f_n+8n^2f_{n-1}\ \ (n=1,2,3,\ldots).
$$
These numbers are now called Franel numbers. Callan \cite{C} found a combinatorial interpretation of the Franel numbers. The Franel numbers play important roles in combinatorics and number theory.  The sequence $(f_n)_{n\gs0}$ is one of the five sporadic sequences (cf. \cite[Section 4]{Z})
which are integral solutions of certain Ap\'ery-like recurrence equations and closely related to the theory of modular forms.
In 2013, Sun \cite{S13} revealed some unexpected connections between the numbers $f_n$ and representations of primes $p\eq1\pmod3$ in the form $x^2+3y^2$ with $x,y\in\mathbb{Z}$, for example, Z.-W. Sun \cite[(1.2)]{S13} showed that
\begin{equation}\label{fp2}
\sum_{k=0}^{p-1}\frac{f_k}{2^k}\equiv\sum_{k=0}^{p-1}\frac{f_k}{(-4)^k}\equiv 2x-\frac{p}{2x}\pmod{p^2},
\end{equation}
and in the same paper, Sun proposed some conjectures involving Franel numbers, one of which is
\begin{conjecture}\label{Conj1.1} Let $p>2$ be a prime. If $p=x^2+3y^2$ with $x,y\in\mathbb{Z}$ and $x\equiv1\pmod 3$, then
$$x\equiv\frac14\sum_{k=0}^{p-1}(3k+4)\frac{f_k}{2^k}\equiv\frac12\sum_{k=0}^{p-1}(3k+2)\frac{f_k}{(-4)^k}\pmod{p^2}.$$
\end{conjecture}
\noindent For more studies on Franel numbers, we refer the readers to (\cite{G,G1,Liu,M,m2021,sunh,sun} and so on).

 In this paper, our first goal is to prove the above conjecture.
 \begin{theorem}\label{Th1.1} Conjecture \ref{Conj1.1} is true.
\end{theorem}
Z.-W. Sun \cite{S13} also gave the following conjecture which is much difficult and complex.
\begin{conjecture}\label{Conj1.2} Let $p>2$ be a prime. If $p\equiv1\pmod3$, then
$$\sum_{k=0}^{p-1}\frac{f_k}{2^k}\equiv\sum_{k=0}^{p-1}\frac{f_k}{(-4)^k}\pmod{p^3}.$$
\end{conjecture}
Our last goal is to prove this conjecture.
\begin{theorem}\label{Th1.2} Conjecture \ref{Conj1.2} is true.
\end{theorem}
We are going to prove Theorem \ref{Th1.1} in Section 2. Section 3 is devoted to proving Theorem \ref{Th1.2}. Our proofs make use of some combinatorial identities which were found by the package \texttt{Sigma} \cite{S} via the software \texttt{Mathematica} and the $p$-adic Gamma function. The proof of Theorem \ref{Th1.2} is somewhat difficult and complex because it is rather convoluted. Throughout this paper, prime $p$ always $\equiv1\pmod3$, so in the following Lemmas $p>5$ or $p>3$ or $p>2$ is the same, we mention it here first.
 \section{Proof of Theorem 1.1}
 \setcounter{lemma}{0}
\setcounter{theorem}{0}
\setcounter{corollary}{0}
\setcounter{remark}{0}
\setcounter{equation}{0}
\setcounter{conjecture}{0}
For a prime $p$, let  $\mathbb{Z}_p$ denote the ring of all $p$-adic integers and let $\mathbb{Z}_p^{\times}:=\{a\in\mathbb{Z}_p:\,a\text{ is prime to }p\}.$
For each $\alpha\in\mathbb{Z}_p$, define the $p$-adic order $\nu_p(\alpha):=\max\{n\in\mathbb{N}:\, p^n\mid \alpha\}$ and the $p$-adic norm $|\alpha|_p:=p^{-\nu_p(\alpha)}$. Define the $p$-adic gamma function $\Gamma_p(\cdot)$ by
$$
\Gamma_p(n)=(-1)^n\prod_{\substack{1\leq k<n\\ (k,p)=1}}k,\qquad n=1,2,3,\ldots,
$$
and
$$
\Gamma_p(\alpha)=\lim_{\substack{|\alpha-n|_p\to 0\\ n\in\mathbb{N}}}\Gamma_p(n),\qquad \alpha\in\mathbb{Z}_p.
$$
In particular, we set $\Gamma_p(0)=1$. Following, we need to use the most basic properties of $\Gamma_p$, and all of them can be found in \cite{Murty02,Robert00}.
For example, we know that
\begin{equation}\label{Gammap}
\frac{\Gamma_p(x+1)}{\Gamma_p(x)}=\begin{cases}-x,&\text{if }|x|_p=1,\\
-1,&\text{if }|x|_p>1.
\end{cases}
\end{equation}
\begin{align}\label{Gammap1xx}
\Gamma_p(1-x)\Gamma_p(x)=(-1)^{a_0(x)},
\end{align}
where $a_0(x)\in\{1,2,\ldots,p\}$ such that $x\equiv a_0(x)\pmod p$. And a property we need here is the fact that for any positive integer $n$,
\begin{equation}\label{pn}
z_1\equiv z_2\pmod{p^n}\ \ \ \mbox{implies}\ \ \ \Gamma_p(z_1)\equiv\Gamma_p(z_2)\pmod{p^n}.
\end{equation}
\begin{lemma}\label{Lem2.1} {\rm (\cite[Lemma 2.2]{S13})} For any $n\in\mathbb{N}$ we have
 \begin{align}\label{2.1}\sum_{k=0}^n\binom{n}k^3z^k=\sum_{k=0}^{\lfloor n/2\rfloor}\binom{n+k}{3k}\binom{2k}k\binom{3k}kz^k(1+z)^{n-2k}
 \end{align}
 and
 \begin{align}\label{2.2}f_n=\sum_{k=0}^n\binom{n+2k}{3k}\binom{2k}k\binom{3k}k(-4)^{n-k}.
 \end{align}
\end{lemma}
For $n,m\in\{1,2,3,\ldots\}$, define
$$
H_n^{(m)}=\sum_{1\leq k\leq n}\frac1{k^m},
$$
these numbers with $m=1$ are often called the classic harmonic numbers. 
 
\noindent Recall that the Bernoulli polynomials are given by
$$B_n(x)=\sum_{k=0}^n\bi nkB_kx^{n-k}\ \ (n=0,1,2,\ldots).$$
\begin{lemma}\label{sunh}{\rm (\cite{s2000,s2008})} Let $p>5$ be a prime. Then
\begin{gather*}
H_{p-1}^{(2)}\equiv0\pmod{p},\ \ H_{\frac{p-1}2}^{(2)}\equiv0\pmod{p},\ \ H_{p-1}\equiv0\pmod{p^2},\\
\frac15H_{\lfloor\frac{p}6\rfloor}^{(2)}\equiv H_{\lfloor\frac{p}3\rfloor}^{(2)}\equiv\frac12\left(\frac{p}3\right)B_{p-2}\left(\frac13\right)\pmod p,\\
H_{\lfloor\frac{p}6\rfloor}\equiv-2q_p(2)-\frac32q_p(3)+pq^2_p(2)+\frac{3p}4q^2_p(3)-\frac{5p}{12}\left(\frac{p}3\right)B_{p-2}\left(\frac13\right)\pmod{p^2},\\
H_{\lfloor\frac{p}3\rfloor}\equiv-\frac32q_p(3)+\frac{3p}4q^2_p(3)-\frac{p}6\left(\frac{p}3\right)B_{p-2}\left(\frac13\right)\pmod{p^2},\\
H_{\frac{p-1}2}\equiv-2q_p(2)+pq^2_p(2)\pmod{p^2},\ H_{\lfloor\frac{p}4\rfloor}^{(2)}\equiv(-1)^{\frac{p-1}2}4E_{p-3}\pmod p,\\
H_{\lfloor\frac{2p}3\rfloor}\equiv-\frac32q_p(3)+\frac{3p}4q^2_p(3)+\frac{p}3\left(\frac{p}3\right)B_{p-2}\left(\frac13\right)\pmod{p^2},
\end{gather*}
where $q_p(a)=(a^{p-1}-1)/p$ stands for the Fermat quotient.
\end{lemma}
\begin{lemma}\label{Lem2.2} Let $p>2$ be a prime and $p\equiv1\pmod3$. If $0\leq j\leq (p-1)/2$, then we have
$$
\binom{3j}j\binom{p+j}{3j+1}\equiv \frac{p}{3j+1}(1-pH_{2j}+pH_j)\pmod{p^3}.
$$
\end{lemma}
\begin{proof} If $0\leq j\leq (p-1)/2$ and $j\neq(p-1)/3$, then we have
\begin{align*}
\binom{3j}j\binom{p+j}{3j+1}&=\frac{(p+j)\cdots(p+1)p(p-1)\cdots(p-2j)}{j!(2j)!(3j+1)}\\
&\equiv\frac{pj!(1+pH_j)(-1)^{2j}(2j)!(1-pH_{2j})}{j!(2j)!(3j+1)}\\
&\equiv\frac{p}{3j+1}(1-pH_{2j}+pH_j)\pmod{p^3}.
\end{align*}
If $j=(p-1)/3$, then by Lemma \ref{sunh}, we have
\begin{align*}
&\binom{p-1}{\frac{p-1}3}\binom{p+\frac{p-1}3}{\frac{p-1}3}\\
\equiv&\left(1-pH_{\frac{p-1}3}+\frac{p^2}2(H_{\frac{p-1}3}^2-H_{\frac{p-1}3}^{(2)})\right)\left(1+pH_{\frac{p-1}3}+\frac{p^2}2(H_{\frac{p-1}3}^2-H_{\frac{p-1}3}^{(2)})\right)\\
\equiv&1-p^2H_{\frac{p-1}3}^{(2)}\equiv1-\frac{p^2}2\left(\frac{p}3\right)B_{p-2}\left(\frac13\right)\pmod{p^3}
\end{align*}
and
$$
1-pH_{\frac{2p-2}3}+pH_{\frac{p-1}3}\equiv1-\frac{p^2}2\left(\frac{p}3\right)B_{p-2}\left(\frac13\right)\pmod{p^3}.
$$
Now the proof of Lemma \ref{Lem2.2} is complete.
\end{proof}
\medskip
\noindent{\it Proof of Theorem 1.1}. With the help of (\ref{2.1}), we have
\begin{align}\label{3k4f}
\sum_{k=0}^{p-1}(3k+4)\frac{f_k}{2^k}&=\sum_{k=0}^{p-1}\frac{3k+4}{2^k}\sum_{j=0}^{\lfloor k/2\rfloor}\binom{k+j}{3j}\binom{2j}j\binom{3j}j2^{k-2j}\notag\\
&=\sum_{j=0}^{(p-1)/2}\frac{\binom{2j}j\binom{3j}j}{4^j}\sum_{k=2j}^{p-1}(3k+4)\binom{k+j}{3j}.
\end{align}
By loading the package \texttt{Sigma} in the software \texttt{Mathematica}, we find the following identity:
$$
\sum_{k=2j}^{n-1}(3k+4)\binom{k+j}{3j}=\frac{9nj+3n+9j+5}{3j+2}\binom{n+j}{3j+1}.$$
Thus, replacing $n$ by $p$ in the above identity and then substitute it into (\ref{3k4f}), we have
\begin{equation*}
\sum_{k=0}^{p-1}(3k+4)\frac{f_k}{2^k}=\sum_{j=0}^{(p-1)/2}\frac{\binom{2j}j\binom{3j}j}{4^j}\frac{9pj+3p+9j+5}{3j+2}\binom{p+j}{3j+1}.
\end{equation*}
Hence we immediately obtain the following result by Lemma \ref{Lem2.2},
\begin{equation}\label{main1}
\sum_{k=0}^{p-1}(3k+4)\frac{f_k}{2^k}\equiv p\sum_{j=0}^{(p-1)/2}\frac{\binom{2j}j}{4^j}\frac{9j+5}{(3j+1)(3j+2)}\pmod{p^2}.
\end{equation}
It is easy to verify that
\begin{align}\label{equ}
p\sum_{j=0}^{(p-1)/2}\frac{\binom{2j}j}{4^j}\frac{9j+5}{(3j+1)(3j+2)}\equiv S_1+S_2\pmod{p^2},
\end{align}
where
\begin{align}\label{s1}
S_1= p\sum_{j=0}^{(p-1)/2}\binom{(p-1)/2}j(-1)^j\left(\frac2{3j+1}+\frac1{3j+2}\right)
\end{align}
and
$$
S_2=\frac{3p+2}{p+1}\left(\binom{(2p-2)/3}{(p-1)/3}4^{(1-p)/3}-\binom{(p-1)/2}{(p-1)/3}\right).
$$
Applying the famous identity
\begin{equation}\label{kx}
\sum_{k=0}^n\binom{n}k\frac{(-1)^k}{k+x}=\frac{n!}{x(x+1)\cdots(x+n)}
\end{equation}
with $x=1/3, n=(p-1)/2$ and $x=2/3,n=(p-1)/2$, we may simplify (\ref{s1}) as
$$
S_1=\frac{4p}{3p-1}\frac{(1)_{(p-1)/2}}{(1/3)_{(p-1)/2}}+\frac{2p}{3p+1}\frac{(1)_{(p-1)/2}}{(2/3)_{(p-1)/2}},
$$
where $(a)_n=a(a+1)\cdots(a+n-1)$ is the rising factorial or the Pochhammer symbol.

In view of (\ref{Gammap1xx}), we have
\begin{align*}
&\frac{4p}{3p-1}\frac{(1)_{\frac{p-1}2}}{(1/3)_{\frac{p-1}2}}=\frac{4p}{3p-1}\frac{\Gamma(\frac{p+1}2)\Gamma(\frac13)}{\Gamma(\frac13+\frac{p-1}2)}=\frac{4p}{3p-1}\frac{(-1)^{\frac{p+1}2}\Gamma_p(\frac{p+1}2)\Gamma_p(\frac13)}{(-1)^{\frac{p-1}2}\frac{p}3\Gamma_p(\frac13+\frac{p-1}2)}\\
&=\frac{12}{1-3p}\frac{\Gamma_p(\frac{p+1}2)\Gamma_p(\frac13)}{\Gamma_p(\frac{p}2-\frac16)}=\frac{12(-1)^{\frac{p-1}6}}{1-3p}\Gamma_p\left(\frac{p+1}2\right)\Gamma_p\left(\frac13\right)\Gamma_p\left(\frac76-\frac{p}2\right),
\end{align*}
where $\Gamma(\cdot)$ is the Gamma function. It is known that for $\alpha,s\in\mathbb{Z}_p$, we have
\begin{align}\label{Gammap2}
\Gamma_p(\alpha+ps)\equiv\Gamma_p(\alpha)+ps\Gamma^{'}_p(\alpha)\pmod{p^2}
\end{align}
and
\begin{align}\label{Gammap'}
\frac{\Gamma^{'}_p(\alpha)}{\Gamma_p(\alpha)}\equiv1+H_{p-\langle-\alpha\rangle_p-1}\pmod{p},
\end{align}
where $\Gamma_p'(x)$ denotes the $p$-adic derivative of $\Gamma_p(x)$, $\langle\alpha\rangle_n$
denotes the least non-negative residue of $\alpha$ modulo $n$, i.e., the integer lying in $\{0,1,\ldots,n-1\}$ such that $\langle\alpha\rangle_n\equiv\alpha\pmod{n}$.

\noindent Therefore modulo $p^2$, we have
\begin{align*}
\frac{4p}{3p-1}\frac{(1)_{\frac{p-1}2}}{(\frac13)_{\frac{p-1}2}}\equiv\frac{12(-1)^{\frac{p-1}6}\Gamma_p\left(\frac{1}2\right)\Gamma_p\left(\frac13\right)\Gamma_p\left(\frac76\right)}{1-3p}\left(1+\frac{p}2(H_{\frac{p-1}2}-H_{\frac{p-7}6})\right).
\end{align*}
In view of (\ref{Gammap}) and (\ref{Gammap1xx}), we have
\begin{align*}
\frac{4p}{3p-1}\frac{(1)_{\frac{p-1}2}}{(\frac13)_{\frac{p-1}2}}\equiv\frac{2(1+3p)\Gamma_p\left(\frac{1}2\right)\Gamma_p\left(\frac13\right)}{\Gamma_p\left(\frac56\right)}\left(1+\frac{p}2(H_{\frac{p-1}2}-H_{\frac{p-7}6})\right)\pmod{p^2}.
\end{align*}
And then by using \cite[Proposition 4.1]{YMK}, we have
$$
\frac{\Gamma_p\left(\frac{1}2\right)\Gamma_p\left(\frac13\right)}{\Gamma_p\left(\frac56\right)}\equiv\frac{\binom{(5p-5)/6}{(p-1)/3}}{\left(1+\frac{p}6(5H_{(5p-5)/6}-3H_{(p-1)/2}-2H_{(p-1)/3})\right)}\pmod{p^2}.
$$
Then with the help of \cite[Theorem 4.12]{YMK} and Lemma \ref{sunh}, we have
\begin{equation}\label{4p}
\frac{4p}{3p-1}\frac{(1)_{\frac{p-1}2}}{(1/3)_{\frac{p-1}2}}\equiv4x+3pxq_p(3)-\frac{p}{x}\pmod{p^2}
\end{equation}
and
\begin{equation}\label{2p}
\frac{2p}{3p+1}\frac{(1)_{(p-1)/2}}{(2/3)_{(p-1)/2}}\equiv\frac{p}{x}\pmod{p^2}.
\end{equation}
\noindent Hence
\begin{equation}\label{s1p2}
S_1\equiv 4x+3pxq_p(3)\pmod{p^2}.
\end{equation}
\begin{lemma}\label{mpt} Let $p>3$ be a prime. For any $p$-adic integer $t$, we have 
\begin{equation}\label{3.1}
 \binom{\frac{p-1}2+pt}{\frac{p-1}3}\equiv\binom{\frac{p-1}2}{\frac{p-1}3}\left(1+pt\left(H_{\frac{p-1}2}-H_{\frac{p-1}6}\right)\right)\pmod{p^2}.
 \end{equation}
 \end{lemma}
\begin{proof} Set $m=(p-1)/2$. It is easy to check that
\begin{align*}
\binom{m+pt}{(p-1)/3}&=\frac{(m+pt)\cdots(m+pt-(p-1)/3+1)}{((p-1)/3)!}\\
&\equiv\frac{m\cdots(m-(p-1)/3+1)}{((p-1)/3)!}(1+pt(H_m-H_{m-(p-1)/3})\\
&=\binom{m}{(p-1)/3}(1+pt(H_m-H_{m-(p-1)/3})\pmod{p^2}.
\end{align*}
So Lemma \ref{mpt} is finished.
\end{proof}
\noindent Now we evaluate $S_2$ modulo $p^2$. It is easy to obtain that
\begin{align}\label{s2}
S_2&\equiv2\left(\binom{-\frac12}{(p-1)/3}-\binom{(p-1)/2}{(p-1)/3}\right)\notag\\
&\equiv-p\binom{(p-1)/2}{(p-1)/3}(H_{(p-1)/2}-H_{(p-1)/6})\notag\\
&\equiv-3pxq_p(3)\pmod{p^2}
\end{align}
with the help of Lemma \ref{sunh}, Lemma \ref{mpt} and \cite[Theorem 4.12]{YMK}.

\noindent Therefore, in view of (\ref{main1}), (\ref{equ}), (\ref{s1p2}) and (\ref{s2}), we immediately get the desired result
$$
\frac14\sum_{k=0}^{p-1}(3k+4)\frac{f_k}{2^k}\equiv x\pmod{p^2}.
$$
On the other hand, we use the equation (\ref{2.2}) to obtain that
\begin{align*}
\sum_{k=0}^{p-1}(3k+2)\frac{f_k}{(-4)^k}&=\sum_{k=0}^{p-1}\frac{3k+2}{(-4)^k}\sum_{j=0}^k\binom{k+2j}{3j}\binom{2j}j\binom{3j}j(-4)^{k-j}\\
&=\sum_{j=0}^{p-1}\frac{\binom{2j}j\binom{3j}j}{(-4)^j}\sum_{k=j}^{p-1}(3k+2)\binom{k+2j}{3j}.
\end{align*}
By using the package \texttt{Sigma} again, we find the following identity:
$$
\sum_{k=j}^{n-1}(3k+2)\binom{k+2j}{3j}=\frac{9nj+3n+1}{3j+2}\binom{n+2j}{3j+1}.
$$
Thus,
\begin{equation}\label{3k2fk}
\sum_{k=0}^{p-1}(3k+2)\frac{f_k}{(-4)^k}=\sum_{j=0}^{p-1}\frac{\binom{2j}j\binom{3j}j\binom{p+2j}{3j+1}}{(-4)^j}\frac{9pj+3p+1}{3j+2}.
\end{equation}
\begin{lemma}\label{p2j} Let $p>3$ be a prime and $p\equiv1\pmod3$. If $0\leq j\leq (p-1)/2$ and $j\neq(p-1)/3$, then 
$$
\binom{3j}j\binom{p+2j}{3j+1}\equiv \frac{p(-1)^j}{3j+1}(1+pH_{2j}-pH_j)\pmod{p^3}.
$$
If $(p+1)/2\leq j\leq p-1$, then
$$
\binom{3j}j\binom{p+2j}{3j+1}\equiv \frac{2p(-1)^j}{3j+1}\pmod{p^2}.
$$
\end{lemma}
\begin{proof}If $0\leq j\leq (p-1)/2$ and $j\neq(p-1)/3$, then we have
\begin{align*}
\binom{3j}j\binom{p+2j}{3j+1}&=\frac{(p+2j)\cdots(p+1)p(p-1)\cdots(p-j)}{j!(2j)!(3j+1)}\\
&\equiv\frac{p(2j)!(1+pH_{2j})(-1)^{j}(j)!(1-pH_j)}{j!(2j)!(3j+1)}\\
&\equiv\frac{p(-1)^j}{3j+1}(1+pH_{2j}-pH_j)\pmod{p^3}.
\end{align*}
If $(p+1)/2\leq j\leq p-1$, then
\begin{align*}
&\binom{3j}j\binom{p+2j}{3j+1}\\
&=\frac{(p+2j)\cdots(2p+1)(2p)(2p-1)\cdots(p+1)p(p-1)\cdots(p-j)}{j!(2j)!(3j+1)}\\
&\equiv\frac{2p^2(2j)\cdots(p+1)(p-1)!(-1)^{j}(j)!}{j!(2j)!(3j+1)}=\frac{2p(-1)^j}{3j+1}\pmod{p^2}.
\end{align*}
Now the proof of Lemma \ref{p2j} is complete.
\end{proof}
It is known that $\binom{2k}k\equiv0\pmod p$ for each $(p+1)/2\leq k\leq p-1$, and it is easy to check that for each $0\leq j\leq (p-1)/2$,
$$
\binom{3j}j\binom{p+2j}{3j+1}\equiv \frac{p(-1)^j}{3j+1}\pmod{p^2}.
$$
These, with (\ref{3k2fk}) yield that
\begin{align}\label{main2}
&\sum_{k=0}^{p-1}(3k+2)\frac{f_k}{(-4)^k}\equiv\sum_{j=0}^{\frac{p-1}2}\frac{\binom{2j}j}{(-4)^j}\frac{p(-1)^j}{3j+1}\frac{9pj+3p+1}{3j+2}+\sum_{j=\frac{p+1}2}^{p-1}\frac{\binom{2j}j}{(-4)^j}\frac{2p(-1)^j}{3j+1}\frac{1}{3j+2}\notag\\
&\equiv\sum_{j=0}^{\frac{p-1}2}\binom{\frac{p-1}2}j\frac{p(-1)^j}{3j+1}\frac{1}{3j+2}+S_3\notag\\
&=p\sum_{j=0}^{\frac{p-1}2}\binom{\frac{p-1}2}j(-1)^j\left(\frac{1}{3j+1}-\frac{1}{3j+2}\right)+S_3\pmod{p^2},
\end{align}
where
\begin{align*}
S_3&=\binom{\frac{2p-2}3}{\frac{p-1}3}\frac1{(p+1)4^{\frac{p-1}3}}-\binom{\frac{p-1}2}{\frac{p-1}3}\frac1{p+1}-\binom{\frac{4p-4}3}{\frac{2p-2}3}\frac1{4^{\frac{2p-2}3}}\\
&=\frac1{p+1}\left(\binom{-1/2}{(p-1)/3}-\binom{\frac{p-1}2}{\frac{p-1}3}\right)-\binom{-1/2}{(2p-2)/3}.
\end{align*}
In the same way of above, with (\ref{kx}), (\ref{4p}), (\ref{2p}), Lemma \ref{sunh} and \cite[Theorem 4.12]{YMK}, we have the following congruence modulo $p^2$
\begin{align}\label{3j13j2}
p\sum_{j=0}^{\frac{p-1}2}\binom{\frac{p-1}2}j(-1)^j\left(\frac{1}{3j+1}-\frac{1}{3j+2}\right)\equiv2x+\frac{3px}2q_p(3)-\frac{3p}{2x}.
\end{align}
Now we evaluate $S_3$. It is easy to see that
\begin{align*}
\binom{-1/2}{(2p-2)/3}&=\frac{(-\frac12)(-\frac12-1)\cdots(-\frac12-\frac{2p-2}3+1)}{(\frac{2p-2}3)!}\\
&=\frac{(\frac12)(\frac32)\cdots(\frac{p}2-1)\frac{p}2(\frac{p}2+1)\cdots(\frac{p}2+\frac{p-7}6)}{(\frac{2p-2}3)!}\\
&=\frac{(\frac{p}2-\frac{p-1}2)\cdots(\frac{p}2-1)\frac{p}2(\frac{p}2+1)\cdots(\frac{p}2+\frac{p-7}6)}{(\frac{2p-2}3)!}\\
&\equiv\frac{(-1)^{\frac{p-1}2}\frac{p}2(\frac{p-1}2)!(\frac{p-7}6)!}{(\frac{2p-2}3)!}=\frac{(-1)^{\frac{p-1}2}3p}{p-1}\frac1{\binom{\frac{2p-2}3}{\frac{p-1}2}}\\
&\equiv\frac{-3p(-1)^{(p-1)/2}}{\binom{\frac{2p-2}3}{\frac{p-1}2}}\pmod{p^2}.
\end{align*}
In view of (\ref{s2}) and \cite[Theorem 4.12]{YMK}, we immediately obtain that
$$
S_3\equiv-\frac{3px}2q_p(3)+\frac{3p}{2x}\pmod{p^2}.
$$
This, with (\ref{main2}) and (\ref{3j13j2}) yields that
$$
\frac12\sum_{k=0}^{p-1}(3k+2)\frac{f_k}{(-4)^k}\equiv x\pmod{p^2}
$$
Now the proof of Theorem \ref{Th1.1} is complete.\qed
\section{Proof of Theorem 1.2}
 \setcounter{lemma}{0}
\setcounter{theorem}{0}
\setcounter{corollary}{0}
\setcounter{remark}{0}
\setcounter{equation}{0}
\setcounter{conjecture}{0}
\noindent {\it Proof of Theorem \ref{Th1.2}}. With the help of (\ref{2.1}), we have
\begin{align}\label{4f}
\sum_{k=0}^{p-1}\frac{f_k}{2^k}&=\sum_{k=0}^{p-1}\frac{1}{2^k}\sum_{j=0}^{\lfloor k/2\rfloor}\binom{k+j}{3j}\binom{2j}j\binom{3j}j2^{k-2j}\notag\\
&=\sum_{j=0}^{(p-1)/2}\frac{\binom{2j}j\binom{3j}j}{4^j}\sum_{k=2j}^{p-1}\binom{k+j}{3j}.
\end{align}
By loading the package \texttt{Sigma} in the software \texttt{Mathematica}, we have the following identity:
$$
\sum_{k=2j}^{n-1}\binom{k+j}{3j}=\binom{n+j}{3j+1}.$$
Thus, replace $n$ by $p$ in the above identity and then substitute it into (\ref{4f}), we have
\begin{equation*}
\sum_{k=0}^{p-1}\frac{f_k}{2^k}=\sum_{j=0}^{(p-1)/2}\frac{\binom{2j}j\binom{3j}j}{4^j}\binom{p+j}{3j+1}.
\end{equation*}
Hence we immediately obtain the following result by Lemma \ref{Lem2.2},
\begin{equation}\label{zmain1}
\sum_{k=0}^{p-1}\frac{f_k}{2^k}\equiv p\sum_{j=0, j\neq\frac{p-1}3}^{(p-1)/2}\frac{\binom{2j}j}{4^j}\frac{1-pH_{2k}+pH_k}{(3j+1)}+S_1\pmod{p^3},
\end{equation}
where $$S_1=\frac{\binom{\frac{2p-2}3}{\frac{p-1}3}\binom{p-1}{\frac{p-1}3}\binom{p+\frac{p-1}3}{p}}{4^{\frac{p-1}3}}=\binom{-\frac12}{\frac{p-1}3}\binom{p-1}{\frac{p-1}3}\binom{p+\frac{p-1}3}{p}.$$
It is easy to verify that
\begin{align*}
&p\sum_{j=0, j\neq\frac{p-1}3}^{(p-1)/2}\frac{\binom{2j}j}{4^j}\frac{1-pH_{2k}+pH_k}{(3j+1)}\\
&\equiv p\sum_{j=0, j\neq\frac{p-1}3}^{(p-1)/2}\frac{\binom{\frac{p-1}2}j(-1)^j(1-pH_{2k}+pH_k)}{(3j+1)\left(1-p\sum_{r=1}^j\frac1{2r-1}\right)}\\
&\equiv p\sum_{j=0}^{(p-1)/2}\frac{\binom{\frac{p-1}2}j(-1)^j\left(1+\frac{p}2H_k\right)}{(3j+1)}-S_2\pmod{p^3},
\end{align*}
where
\begin{align*}
S_2= \binom{\frac{p-1}2}{\frac{p-1}3}\left(1+\frac{p}2H_{\frac{p-1}3}\right).
\end{align*}
So
\begin{align}\label{main}
\sum_{k=0}^{p-1}\frac{f_k}{2^k}\equiv p\sum_{j=0}^{(p-1)/2}\frac{\binom{\frac{p-1}2}j(-1)^j\left(1+\frac{p}2H_k\right)}{(3j+1)}+S_1-S_2\pmod{p^3}.
\end{align}
It is easy to see that
\begin{align}\label{2p3p-11p-12}
\frac{2p}{3p-1}\frac{(1)_{\frac{p-1}2}}{(\frac13)_{\frac{p-1}2}}=\frac{(\frac{p-1}2)!}{\frac13\cdots(\frac{p}3-1)(\frac{p}3+1)\cdots(\frac{p}3+\frac{p-1}6)}\equiv\binom{\frac{p-1}2}{\frac{p-1}3}\pmod{p}.
\end{align}
On the other hand. We have
\begin{align*}
\sum_{k=0}^{p-1}\frac{f_k}{(-4)^k}&=\sum_{k=0}^{p-1}\frac1{(-4)^k}\sum_{j=0}^k\binom{k+2j}{3j}\binom{2j}j\binom{3j}j(-4)^{k-j}\\
&=\sum_{j=0}^{p-1}\frac{\binom{2j}{j}\binom{3j}j}{(-4)^{j}}\sum_{k=j}^{p-1}\binom{k+2j}{3j}=\sum_{j=0}^{p-1}\frac{\binom{2j}{j}\binom{3j}j}{(-4)^{j}}\binom{p+2j}{3j+1}.
\end{align*}
So by Lemma \ref{p2j} and the fact that for each $0\leq k\leq (p-1)/2$,
\begin{gather*}
\frac{\binom{2k}k}{(-4)^k}\equiv\frac{\binom{(p-1)/2}k}{(1-p\sum_{j=1}^k\frac1{2j-1})}\pmod{p^2},
j\binom{2j}j\binom{2p-2j}{p-j}\equiv2p\pmod{p^2}
\end{gather*}
 for each $(p+1)/2\leq j\leq p-1$, we have the following congruence modulo $p^3$,
\begin{align*}
&\sum_{k=0}^{p-1}\frac{f_k}{(-4)^k}-S_3\equiv p\sum_{j=0\atop j\neq\frac{p-1}3}^{\frac{p-1}2}\frac{\binom{2j}{j}(1+pH_{2j}-pH_j)}{(3j+1)4^{j}}+2p\sum_{j=\frac{p+1}2}^{p-1}\frac{\binom{2j}{j}}{(3j+1)4^{j}}\\
&\equiv\sum_{j=0\atop j\neq\frac{p-1}3}^{\frac{p-1}2}\frac{p(-1)^j\binom{\frac{p-1}2}{j}\left(1+2pH_{2j}-\frac32pH_j\right)}{3j+1}+\sum_{j=\frac{p+1}2}^{p-1}\frac{4p^2 }{4^{j}(3j+1)j\binom{2p-2j}{p-j}}\\
&\equiv\sum_{j=0}^{\frac{p-1}2}\frac{p(-1)^j\binom{\frac{p-1}2}{j}\left(1+2pH_{2j}-\frac32pH_j\right)}{3j+1}+\sum_{j=1}^{\frac{p-1}2}\frac{p^24^j }{(3j-1)j\binom{2j}{j}}-S_4,
\end{align*}
where
\begin{gather*}
S_3=\frac{\binom{\frac{2p-2}3}{\frac{p-1}3}\binom{p-1}{\frac{p-1}3}\binom{p+\frac{2p-2}{3}}{p}}{(-4)^{\frac{p-1}3}}=\binom{-\frac12}{\frac{p-1}3}\binom{p-1}{\frac{p-1}3}\binom{p+\frac{2p-2}{3}}{p},\\
S_4=\binom{\frac{p-1}2}{\frac{p-1}3}\left(1+2pH_{\frac{2p-2}3}-\frac32pH_{\frac{p-1}3}\right).
\end{gather*}
Hence we have
\begin{align}\label{zhuyao}
&\sum_{k=0}^{p-1}\frac{f_k}{(-4)^k}-\sum_{k=0}^{p-1}\frac{f_k}{2^k}\notag\\
\equiv&2p^2\sum_{j=0}^{\frac{p-1}2}\frac{\binom{\frac{p-1}2}{j}(-1)^j(H_{2j}-H_j)}{3j+1}+S_5+\sum_{j=1}^{\frac{p-1}2}\frac{p^24^j }{(3j-1)j\binom{2j}{j}}\pmod{p^3},
\end{align}
where $$S_5=S_3-S_4+S_2-S_1.$$
By \texttt{Sigma}, we can find and prove the following identity:
\begin{align}\label{zhuid}
&\sum_{j=0}^{n}\frac{2\binom{n}{j}(-1)^j(H_{2j}-H_j)}{3j+1}\notag\\
&=\frac1{3n+1}\prod_{k=1}^n\frac{3k}{3k-2}\left(\sum_{k=1}^n\frac1k\prod_{j=1}^k\frac{3j-2}{3j}-\sum_{k=1}^n\frac1k\prod_{j=1}^k\frac{2(3j-2)}{3(2j-1)}\right)\notag\\
&=\frac{(1)_{n}}{(3n+1)\left(\frac13\right)_n}\left(\sum_{k=1}^n\frac{\left(\frac13\right)_k}{k(1)_k}-\sum_{k=1}^n\frac{\left(\frac13\right)_k}{k\left(\frac12\right)_k}\right).
\end{align}
In view of \cite[Lemma 3.1]{sijnt} and Lemma \ref{sunh}, we have
\begin{align}\label{-1/31}
\sum_{k=1}^{\frac{p-1}2}\frac{\left(\frac13\right)_k}{k(1)_k}=\sum_{k=1}^{\frac{p-1}2}\frac{\binom{-1/3}{k}}{k\binom{-1}k}\equiv\frac32q_p(3)-\frac{3p}4q^2_p(3)
-\frac{p}3\sum_{k=1}^{\frac{p-1}3}\frac{4^k}{k^2\binom{2k}k}\pmod{p^2}.
\end{align}
\begin{align}\label{-1/3-1/2}
\sum_{k=1}^{\frac{p-1}2}\frac{\left(\frac13\right)_k}{k\left(\frac12\right)_k}&=\sum_{k=1}^{\frac{p-1}2}\frac{\binom{-1/3}{k}}{k\binom{-1/2}k}\equiv\frac{4p}3(-1)^{\frac{p-1}2}E_{p-3}+\frac32q_p(3)-\frac{3p}4q^2_p(3)\notag\\
&-\frac{2p}3(-1)^{\frac{p-1}2}\sum_{k=1}^{\frac{p-1}3}\frac{4^k}{(2k-1)k\binom{2k}k}\pmod{p^2}.
\end{align}
It is easy to check that
\begin{align}\label{2k-1k}
\sum_{k=1}^{\frac{p-1}3}\frac{4^k}{(2k-1)k\binom{2k}k}=2\sum_{k=1}^{\frac{p-1}3}\frac{4^k}{(2k-1)\binom{2k}k}-\sum_{k=1}^{\frac{p-1}3}\frac{4^k}{k\binom{2k}k}.
\end{align}
And by \cite[(6)]{swz}, we have
\begin{equation}\label{cyid}
\frac1{\binom{n+1+k}k}=(n+1)\sum_{r=0}^n\binom{n}r(-1)^r\frac{1}{k+r+1}.
\end{equation}
\begin{align}\label{p-132k-1}
&2\sum_{k=1}^{\frac{p-1}3}\frac{4^k}{(2k-1)\binom{2k}k}\equiv2\sum_{k=1}^{\frac{p-1}3}\frac{(-1)^k}{(2k-1)\binom{\frac{p-1}2}k}\equiv(-1)^{\frac{p+1}2}\sum_{k=\frac{p-1}6}^{\frac{p-3}2}\frac{(-1)^k}{(k+1)\binom{\frac{p-1}2}{k}}\notag\\
&=(-1)^{\frac{p+1}2}\left(\sum_{k=0}^{\frac{p-3}2}\frac{(-1)^k}{(k+1)\binom{\frac{p-1}2}{k}}-\sum_{k=0}^{\frac{p-7}6}\frac{(-1)^k}{(k+1)\binom{\frac{p-1}2}{k}}\right)\pmod{p}.
\end{align}
By \texttt{Sigma}, we find the following identity which can be proved by induction on $n$:
\begin{equation*}
\sum_{k=0}^{n}\frac{(-1)^k}{(k+1)\binom{n}{k}}=\frac{2(-1)^n-1}{n+1}-(n+1)H_n^{(2)}-2(n+1)\sum_{k=1}^n\frac{(-1)^k}{k^2}.
\end{equation*}
So by setting $n=(p-1)/2$ in the above identity, we have
\begin{align}\label{p-32}
\sum_{k=0}^{\frac{p-3}2}\frac{(-1)^k}{(k+1)\binom{\frac{p-1}2}{k}}\equiv2\left((-1)^{\frac{p-1}2}-1\right)-(-1)^{\frac{p-1}2}2E_{p-3}\pmod p.
\end{align}
And by (\ref{cyid}), we have
\begin{align*}
&\sum_{k=0}^{\frac{p-7}6}\frac{(-1)^k}{(k+1)\binom{\frac{p-1}2}{k}}\equiv\sum_{k=0}^{\frac{p-7}6}\frac{1}{(k+1)\binom{\frac{p-1}2+k}{k}}\\
&=\sum_{k=0}^{\frac{p-7}6}\frac{1}{k+1}\frac{p-1}2\sum_{r=0}^{\frac{p-3}2}\binom{\frac{p-3}2}{r}(-1)^r\frac1{k+r+1}\\
&\equiv-\frac12\sum_{k=1}^{\frac{p-1}6}\frac{1}{k}\sum_{r=0}^{\frac{p-3}2}\binom{\frac{p-3}2}{r}(-1)^r\frac1{k+r}\\
&=-\frac12H_{\frac{p-1}6}^{(2)}-\frac12\sum_{r=1}^{\frac{p-3}2}\frac{(-1)^r}r\binom{\frac{p-3}2}r\sum_{k=1}^{\frac{p-1}6}\left(\frac1k-\frac1{k+r}\right)\pmod p.
\end{align*}
It is easy to check that
$$
H_{\frac{p-1}6}-\sum_{k=1}^{\frac{p-1}6}\frac1{k+r}\equiv-\sum_{k=1}^r\frac1{k(6k-1)}\pmod p.
$$
And by \texttt{Sigma} again, we have
$$
\sum_{r=1}^{n}\frac{(-1)^r}r\binom{n}r\sum_{k=1}^r\frac1{k(6k-1)}=H_n^{(2)}-\sum_{k=1}^n\frac{(1)_k}{k\left(\frac56\right)_k}.
$$
So by Lemma (\ref{sunh}) and \cite{YMK}, we have
\begin{align*}
&\sum_{k=0}^{\frac{p-7}6}\frac{(-1)^k}{(k+1)\binom{\frac{p-1}2}{k}}\\
&\equiv\frac{(-1)^{\frac{p-1}2}}{x}-2-\frac54\left(\frac{p}3\right)B_{p-2}\left(\frac13\right)-\frac12\sum_{k=1}^{\frac{p-1}2}\frac{(-1)^k}{k^2\binom{-\frac56}{k}}\pmod p.
\end{align*}
Again, by (\ref{cyid}), we have
\begin{align*}
&\sum_{k=1}^{\frac{p-1}2}\frac{(-1)^k}{k^2\binom{-\frac56}{k}}=-\frac65\sum_{k=1}^{\frac{p-1}2}\frac{(-1)^k}{k\binom{-\frac{11}6}{k-1}}\equiv\frac65\sum_{k=0}^{\frac{p-3}2}\frac{(-1)^k}{(k+1)\binom{\frac{5p-11}6}{k}}\\
&=\frac65\sum_{k=0}^{\frac{p-3}2}\frac{1}{(k+1)\binom{\frac{p+5}6+k}{k}}=\frac65\sum_{k=0}^{\frac{p-3}2}\frac1{k+1}\frac{p+5}6\sum_{r=0}^{\frac{p-1}6}(-1)^r\binom{\frac{p-1}6}r\frac1{k+1+r}\\
&\equiv\sum_{k=1}^{\frac{p-1}2}\frac1{k}\sum_{r=0}^{\frac{p-1}6}(-1)^r\binom{\frac{p-1}6}r\frac1{k+r}\\
&=H_{\frac{p-1}2}^{(2)}+\sum_{r=1}^{\frac{p-1}6}\frac{(-1)^r}r\binom{\frac{p-1}6}r\sum_{k=1}^{\frac{p-1}2}\left(\frac1k-\frac1{k+r}\right)\pmod p.
\end{align*}
Also it is easy to see that
$$
H_{\frac{p-1}2}-\sum_{k=1}^{\frac{p-1}2}\frac1{k+r}\equiv-\sum_{k=1}^r\frac1{k(2k-1)}\pmod p.
$$
And by \texttt{Sigma}, we have
$$
\sum_{r=1}^{n}\frac{(-1)^r}r\binom{n}r\sum_{k=1}^r\frac1{k(2k-1)}=H_n^{(2)}-\sum_{k=1}^n\frac{4^k}{k^2\binom{2k}k}.
$$
So by Lemma \ref{sunh}, we have
$$
\sum_{k=1}^{\frac{p-1}2}\frac{(-1)^k}{k^2\binom{-\frac56}{k}}\equiv\sum_{k=1}^{\frac{p-1}6}\frac{4^k}{k^2\binom{2k}k}-\frac52\left(\frac{p}3\right)B_{p-2}\left(\frac13\right)\pmod p.
$$
Hence
\begin{equation*}
\sum_{k=0}^{\frac{p-7}6}\frac{(-1)^k}{(k+1)\binom{\frac{p-1}2}{k}}\equiv\frac{(-1)^{\frac{p-1}2}}{x}-2-\frac12\sum_{k=1}^{\frac{p-1}6}\frac{4^k}{k^2\binom{2k}k}\pmod p.
\end{equation*}
This, with (\ref{p-132k-1}) and (\ref{p-32}) yields that
\begin{align}\label{p-13}
&2\sum_{k=1}^{\frac{p-1}3}\frac{4^k}{(2k-1)}\notag\\
\equiv&-2+\frac1x+2E_{p-3}-\frac12(-1)^{\frac{p-1}2}\sum_{k=1}^{\frac{p-1}6}\frac{4^k}{k^2\binom{2k}k}\pmod p.
\end{align}
By \texttt{Sigma}, we have
\begin{equation}\label{heng}
\sum_{k=1}^n\frac{4^k}{k\binom{2k}k}=-2+2\frac{4^n}{\binom{2n}n}.
\end{equation}
So by \cite{YMK}, we have
$$
\sum_{k=1}^{\frac{p-1}3}\frac{4^k}{k\binom{2k}k}\equiv-2+\frac{2}{\binom{\frac{p-1}2}{\frac{p-1}3}}\equiv-2+\frac1x\pmod p.
$$
This, with (\ref{2k-1k}) and (\ref{p-13}) yields that
\begin{align*}
\sum_{k=1}^{\frac{p-1}3}\frac{4^k}{(2k-1)k\binom{2k}k}\equiv2E_{p-3}-\frac12(-1)^{\frac{p-1}2}\sum_{k=1}^{\frac{p-1}6}\frac{4^k}{k^2\binom{2k}k}\pmod p.
\end{align*}
This with (\ref{-1/3-1/2}) yields that
\begin{align}\label{p-121312}
\sum_{k=1}^{\frac{p-1}2}\frac{\left(\frac13\right)_k}{k\left(\frac12\right)_k}\equiv\frac32q_p(3)-\frac{3p}4q^2_p(3)+\frac{p}3\sum_{k=1}^{\frac{p-1}6}\frac{4^k}{k^2\binom{2k}k}\pmod{p^2}.
\end{align}
So by (\ref{-1/31}), we have
$$
\sum_{k=1}^{\frac{p-1}2}\frac{\left(\frac13\right)_k}{k(1)_k}-\sum_{k=1}^{\frac{p-1}2}\frac{\left(\frac13\right)_k}{k\left(\frac12\right)_k}\equiv-\frac{p}3\left(\sum_{k=1}^{\frac{p-1}3}\frac{4^k}{k^2\binom{2k}k}+\sum_{k=1}^{\frac{p-1}6}\frac{4^k}{k^2\binom{2k}k}\right)\pmod{p^2}.
$$
Therefore, by (\ref{zhuid}) and (\ref{2p3p-11p-12}), we have
\begin{align}\label{diyige}
&2p^2\sum_{j=0}^{\frac{p-1}2}\frac{\binom{\frac{p-1}2}{j}(-1)^j(H_{2j}-H_j)}{3j+1}\notag\\
&\equiv-\frac{p^2}3\binom{\frac{p-1}2}{\frac{p-1}3}\left(\sum_{k=1}^{\frac{p-1}3}\frac{4^k}{k^2\binom{2k}k}+\sum_{k=1}^{\frac{p-1}6}\frac{4^k}{k^2\binom{2k}k}\right)\pmod{p^3}.
\end{align}
Now we evaluate the second sum in the right side of (\ref{zhuyao}). It is easy to see that
\begin{equation}\label{3j-1}
\sum_{j=1}^{\frac{p-1}2}\frac{4^j }{(3j-1)j\binom{2j}{j}}=3\sum_{j=1}^{\frac{p-1}2}\frac{4^j }{(3j-1)\binom{2j}{j}}-\sum_{j=1}^{\frac{p-1}2}\frac{4^j }{j\binom{2j}{j}}.
\end{equation}
It is easy to see from (\ref{heng}) that
\begin{align}\label{easy}
\sum_{j=1}^{\frac{p-1}2}\frac{4^j }{j\binom{2j}{j}}\equiv-2+2(-1)^{\frac{p-1}2}\pmod p.
\end{align}
Now we consider the first sum of the right side in (\ref{3j-1}).
$$
\sum_{j=1}^{\frac{p-1}2}\frac{4^j }{(3j-1)\binom{2j}{j}}=\sum_{j=1}^{\frac{p-1}3}\frac{4^j }{(3j-1)\binom{2j}{j}}+\sum_{j=\frac{p+2}3}^{\frac{p-1}2}\frac{4^j }{(3j-1)\binom{2j}{j}}.
$$
The following identity is very import to us:
\begin{equation}\label{important}
\sum_{k=1}^n\frac{4^k}{(k+n)\binom{2k}k}=-2+2\frac{4^n}{\binom{2n}n}-\frac{n\binom{2n}n}{4^n}\sum_{k=1}^n\frac{4^k}{k^2\binom{2k}k}.
\end{equation}
This, with \cite{YMK} yields that
\begin{align}\label{p-133j-1}
&3\sum_{j=1}^{\frac{p-1}3}\frac{4^j }{(3j-1)\binom{2j}{j}}\equiv\sum_{j=1}^{\frac{p-1}3}\frac{4^j }{\left(j+\frac{p-1}3\right)\binom{2j}{j}}\notag\\
\equiv&-2+\frac{2}{\binom{-1/2}{\frac{p-1}3}}+\frac13\binom{-1/2}{\frac{p-1}3}\sum_{k=1}^{\frac{p-1}3}\frac{4^k}{k^2\binom{2k}k}\notag\\
&\equiv-2+\frac1x+\frac13\binom{\frac{p-1}2}{\frac{p-1}3}\sum_{k=1}^{\frac{p-1}3}\frac{4^k}{k^2\binom{2k}k}\pmod{p}.
\end{align}
It is easy to check that And by (\ref{important}), we have
\begin{align}\label{p+2/3}
&3\sum_{j=\frac{p+2}3}^{\frac{p-1}2}\frac{4^j}{(3j-1)\binom{2j}{j}}\equiv3\sum_{j=0}^{\frac{p-7}6}\frac{(-1)^{\frac{p-1}2-j}}{(3(\frac{p-1}2-j)-1)\binom{\frac{p-1}2}{j}}\notag\\
&\equiv6(-1)^{\frac{p+1}2}\sum_{j=0}^{\frac{p-7}6}\frac{4^j }{(6j+5)\binom{2j}{j}}\equiv(-1)^{\frac{p+1}2}\sum_{j=0}^{\frac{p-7}6}\frac{(-1)^j }{(j+\frac{p+5}6)\binom{\frac{p-1}2}{j}}\notag\\
&\equiv\frac65(-1)^{\frac{p+1}2}+(-1)^{\frac{p+1}2}\sum_{j=1}^{\frac{p+5}6}\frac{4^j }{(j+\frac{p+5}6)\binom{2j}{j}}+\frac3{\binom{\frac{p-1}2}{\frac{p-1}3}}\pmod p.
\end{align}
By (\ref{important}) and \cite{YMK}, we have
$$
\sum_{j=1}^{\frac{p+5}6}\frac{4^j }{(j+\frac{p+5}6)\binom{2j}{j}}\equiv-\frac{16}5+\frac{5(-1)^{\frac{p-1}6}}{2x}-\frac{(-1)^{\frac{p-1}6}}{6x}\sum_{k=1}^{\frac{p-1}6}\frac{4^k}{k^2\binom{2k}k}\pmod p.
$$
This, with (\ref{p+2/3}) yields that
\begin{align*}
3\sum_{j=\frac{p+2}3}^{\frac{p-1}2}\frac{4^j}{(3j-1)\binom{2j}{j}}\equiv2(-1)^{\frac{p-1}2}-\frac1x+\frac13\binom{\frac{p-1}2}{\frac{p-1}3}\sum_{k=1}^{\frac{p-1}6}\frac{4^k}{k^2\binom{2k}k}\pmod p.
\end{align*}
Combining this with (\ref{p-133j-1}), we have
\begin{align*}
&3\sum_{j=1}^{\frac{p-1}2}\frac{4^j }{(3j-1)\binom{2j}{j}}\\
&\equiv-2+2(-1)^{\frac{p-1}2}+\frac13\binom{\frac{p-1}2}{\frac{p-1}3}\left(\sum_{k=1}^{\frac{p-1}3}\frac{4^k}{k^2\binom{2k}k}+\sum_{k=1}^{\frac{p-1}6}\frac{4^k}{k^2\binom{2k}k}\right)\pmod p.
\end{align*}
Thus, by (\ref{3j-1}) and (\ref{easy}), we have
$$
\sum_{j=1}^{\frac{p-1}2}\frac{4^j }{(3j-1)j\binom{2j}{j}}\equiv\frac13\binom{\frac{p-1}2}{\frac{p-1}3}\left(\sum_{k=1}^{\frac{p-1}3}\frac{4^k}{k^2\binom{2k}k}+\sum_{k=1}^{\frac{p-1}6}\frac{4^k}{k^2\binom{2k}k}\right)\pmod p.
$$
This, with (\ref{zhuyao}) and (\ref{diyige}) yields that
\begin{equation}\label{zhup3}
\sum_{k=0}^{p-1}\frac{f_k}{(-4)^k}-\sum_{k=0}^{p-1}\frac{f_k}{2^k}\equiv S_5\pmod{p^3}.
\end{equation}
While
\begin{align*}
S_5=&\binom{-\frac12}{\frac{p-1}3}\binom{p-1}{\frac{p-1}3}\left(\binom{p+\frac{2p-2}3}{\frac{p-1}3}-\binom{p+\frac{p-1}3}{\frac{p-1}3}\right)\\
&+2p\binom{\frac{p-1}2}{\frac{p-1}3}\left(H_{\frac{p-1}3}-H_{\frac{2p-2}3}\right).
\end{align*}
It is easy to check that
$$
\binom{p+\frac{2p-2}3}{\frac{p-1}3}\equiv1+pH_{\frac{2p-2}3}+\frac{p^2}2\left(H_{\frac{2p-2}3}^2-H_{\frac{2p-2}3}^{(2)}\right)\pmod{p^3}
$$
and
$$
\binom{p+\frac{p-1}3}{\frac{p-1}3}\equiv1+pH_{\frac{p-1}3}+\frac{p^2}2\left(H_{\frac{p-1}3}^2-H_{\frac{p-1}3}^{(2)}\right)\pmod{p^3}.
$$
So by Lemma \ref{sunh} and the fact that $H_{p-1-k}^{(2)}\equiv-H_k^{(2)}\pmod p$ for each $0\leq k\leq p-1$, we have
\begin{align*}
\binom{p+\frac{2p-2}3}{\frac{p-1}3}-\binom{p+\frac{p-1}3}{\frac{p-1}3}&\equiv p(H_{\frac{2p-2}3}-H_{\frac{p-1}3})+\frac{p^2}2(H_{\frac{p-1}3}^{(2)}-H_{\frac{2p-2}3}^{(2)})\\
&\equiv p^2\left(\frac{p}3\right)B_{p-2}\left(\frac13\right)\pmod{p^3}
\end{align*}
and
$$
2p\left(H_{\frac{p-1}3}-H_{\frac{2p-2}3}\right)\equiv-p^2\left(\frac{p}3\right)B_{p-2}\left(\frac13\right)\pmod{p^3}.
$$
So by $\binom{-\frac12}{\frac{p-1}3}\equiv\binom{\frac{p-1}2}{\frac{p-1}3}\pmod{p}$ and $\binom{p-1}{\frac{p-1}3}\equiv(-1)^{\frac{p-1}3}=1\pmod p$, we can immediately obtain that
$$
S_5\equiv0\pmod{p^3}.
$$
This, with (\ref{zhup3}) yields that
$$
\sum_{k=0}^{p-1}\frac{f_k}{(-4)^k}\equiv\sum_{k=0}^{p-1}\frac{f_k}{2^k}\pmod{p^3}.
$$
Now the proof of Theorem \ref{Th1.2} is complete.\qed

\Ack. The first author is funded by the National Natural Science Foundation of China (12001288) and China Scholarship Council (202008320187).

     \end{document}